\documentclass{article}
\usepackage{amssymb}
\usepackage{amsmath}
\usepackage{amsthm}
\usepackage[pdftex]{graphicx}
\usepackage{tikz}
\usepackage{mathtools}
\usepackage[backend=biber,style=numeric,maxcitenames=10,mincitenames=10,maxbibnames=1000,sorting=nty,isbn=false,url=false,eprint=true]{biblatex}
\usepackage{geometry}
\usepackage{here}
\usepackage[affil-it]{authblk}
\usepackage{diagbox}

\DeclareFieldFormat{doi}{%
  \mkbibacro{DOI}\addcolon\space
  \ifhyperref
    {\href{https://doi.org/#1}{\nolinkurl{https://doi.org/#1}}}
    {\nolinkurl{https://doi.org/#1}}}

\renewbibmacro{in:}{}
\DeclareMathAlphabet{\mymathbb}{U}{BOONDOX-ds}{m}{n}

\newcommand{\range}{\operatorname{ran}}
\newcommand{\dom}{\operatorname{dom}}

\newcommand{\frakc}{\mathfrak{c}}
\newcommand{\frakb}{\mathfrak{b}}
\newcommand{\frakd}{\mathfrak{d}}

\newcommand{\Pow}{\mathcal{P}}
\newcommand{\non}{\operatorname{non}}

\newcommand{\add}{\operatorname{add}}

\newcommand{\nul}{\mathcal{N}}
\newcommand{\meager}{\mathcal{M}}
\newcommand{\cf}{\operatorname{cf}}

\newcommand{\scrA}{\mathcal{A}}

\newcommand{\ZFC}{\mathrm{ZFC}}

\newcommand{\fraks}{\mathfrak{s}}
\newcommand{\frakr}{\mathfrak{r}}
\newcommand{\sI}{\mathfrak{s}_\mathrm{game}^{\mathrm{I}}}
\newcommand{\sII}{\mathfrak{s}_\mathrm{game}^{\mathrm{II}}}

\newcommand{\sIstar}{\mathfrak{s}_\mathrm{game^\ast}^{\mathrm{I}}}
\newcommand{\sIIstar}{\mathfrak{s}_\mathrm{game^\ast}^{\mathrm{II}}}

\newcommand{\append}{{}^\frown}

\newcommand{\zero}{\mymathbb{0}}
\newcommand{\one}{\mymathbb{1}}

\newcommand{\seq}[1]{{\langle#1\rangle}}

\DeclarePairedDelimiter\abs{\lvert}{\rvert}
\DeclarePairedDelimiterX{\norm}[1]{\lVert}{\rVert}{#1}

\theoremstyle{definition}
\newtheorem{thm}{Theorem}[section]
\newtheorem*{thm*}{Theorem}
\newtheorem{defi}[thm]{Definition}
\newtheorem*{defi*}{Definition}
\newtheorem{lem}[thm]{Lemma}
\newtheorem*{lem*}{Lemma}
\newtheorem{fact}[thm]{Fact}
\newtheorem*{fact*}{Fact}

\newtheorem{rmk}[thm]{Remark}
\newtheorem*{rmk*}{Remark}
\newtheorem{cor}[thm]{Corollary}
\newtheorem*{cor*}{Corollary}
\newtheorem*{convention*}{Convention}
\newtheorem{question}{Question}

\renewcommand{\P}{\mathbb{P}}
\newcommand{\pst}{\mathbb{P}^*}
\newcommand{\pstt}{\mathbb{P}^{**}}
\newcommand{\qd}{\dot{\mathbb{Q}}}
\renewcommand{\a}{\alpha}

\newcommand{\on}{\mathpunct{\upharpoonright}}
\newcommand{\ooo}{[{\omega}]^\omega}
\DeclareMathOperator{\cfi}{cf}
\newcommand{\str}{\mathrm{Str}}
\newcommand{\rp}{\sqsubset^\mathrm{r}}
\newcommand{\rsp}{\sqsubset^\mathrm{s}}
\newcommand{\rs}{\mathbf{R}^*}
\newcommand{\rss}{\mathbf{R}^{**}}
\newcommand{\fstr}{\mathrm{FinStr}}
\newcommand{\tri}{\triangleleft_{*}}
\newcommand{\trii}{\triangleleft_{**}}

\newcommand{\nonm}{\non(\mathcal{M})}

\newcommand{\nonn}{\non(\mathcal{N})}

\newcommand{\sIwstar}{\mathfrak{s}_{\mathrm{game^{\ast\ast}}}^\mathrm{I}}
\newtheorem{mainlem}[thm]{Main Lemma}
\newcommand{\sqtwo}{2^{<\omega}}

\newtheorem{subque}[thm]{Question}

\newcommand{\sIIwstar}{\mathfrak{s}_{\mathrm{game^{\ast\ast}}}^\mathrm{II}}

\newcounter{enuAlph}

\newtheorem{teorema}[enuAlph]{Theorem}

\usepackage[labelfont=bf, labelsep=space, font=small]{caption}

\title{Game-theoretic variants of splitting number}

\author{Jorge Antonio Cruz Chapital}
\affil{Department of Mathematics, University of Toronto, 27 King's College Cir, Toronto, Canada. E-mail: chapi@matmor.unam.mx}

\author{Tatsuya Goto\thanks{Supported by JSPS KAKENHI Grant Number JP22J20021}}
\affil{Graduate School of System Informatics, Kobe University, 1-1 Rokkodai, Nada-ku, 657-8501 Kobe, Japan. E-mail: 202x603x@stu.kobe-u.ac.jp}

\author{Yusuke Hayashi\thanks{Supported by JST SPRING, Japan Grant Number JPMJSP2148}}
\affil{Graduate School of System Informatics, Kobe University, 1-1 Rokkodai, Nada-ku, 657-8501 Kobe, Japan. E-mail: 219x504x@stu.kobe-u.ac.jp}

\author{Takashi Yamazoe\thanks{Supported by JST SPRING, Japan Grant Number JPMJSP2148}}
\affil{Graduate School of System Informatics, Kobe University, 1-1 Rokkodai, Nada-ku, 657-8501 Kobe, Japan. E-mail: 212x502x@stu.kobe-u.ac.jp}

\date{\today}

\begin{document}
	\maketitle

    %\begin{abstract}
    %   We investigate game-theoretic variants of the splitting number, which is one of the famous cardinal invariants of the continuum. 
     %   By using games on splitting families, we find a new cardinal invariant that differs from previously studied cardinal invariants.
    %\end{abstract}

    %\begin{abstract}
    %    We study game-theoretic variants of the splitting number $\mathfrak{s}$. We prove that some of them are equal to classical cardinal invariants and two of them are not. Moreover, though the two numbers share almost the same rule of the game, we show that they can take distinct values, and hence the slight difference of the rule is actually crucial in this sense.
    %\end{abstract}

    \begin{abstract}
        We consider combining the definition of a cardinal invariant and the notion of an infinite game. We focus on the splitting number $\mathfrak{s}$ since the corresponding cardinal invariants behave in an interesting way. We introduce three kinds of games as reasonable realizations of the combination of the notions of splitting and infinite games. Then, we consider two cardinal invariants for each game, so we define six numbers. We prove that three of them are equal to the size of the continuum $\mathfrak{c}$ and one of them is equal to the $\sigma$-splitting number $\mathfrak{s}_\sigma$, which is defined as the minimum size of a $\sigma$-splitting family. On the other hand, we show that the remaining two numbers are consistently different from $\mathfrak{c}$, $\mathfrak{s}$ and $\mathfrak{s}_\sigma$. Moreover, though the two numbers share almost the same rule of the game, we prove that they can take distinct values from each other, and hence the slight difference of the rule is actually crucial in this sense.
    \end{abstract}

    \section*{Acknowledgements}

    The authors are grateful to J\"{o}rg Brendle, Osvaldo Guzm\'an and Michael Hru\v{s}\'ak for their helpful comments.
    Remark \ref{rmk:preservingssigma} is based on a suggestion from Diego A. Mejía.
    This work was supported by JSPS KAKENHI Grant Number JP22J20021 and JST SPRING, Japan Grant Number JPMJSP2148.

    \section{Introduction}\label{sec:intro}

    %\subsubsection{Introduction for only splitting* game (written by Takashi)}

    Let us consider combining the definition of a classical cardinal invariant and the notion of an infinite game.
    Game-theoretic considerations of cardinal invariants can be found in \cite{kada2000more}, \cite{brendle2019construction}, and \cite{meagersetsinfinite}, but our approach differs from these.
    The first three authors (Cruz Chapital, Goto, Hayashi) studied game-theoretic variants of some cardinal invariants, such as $\frakb$, $\frakd$, $\frakr$ and $\add(\nul)$, and they will present these results in a forthcoming paper.
    In this paper, along the way of that paper, we focus on the splitting number $\mathfrak{s}$ since the corresponding cardinal invariants 
%$\sIstar$ and $\sIwstar$ 
behave in an interesting way.
We introduce three kinds of games of length $\omega$ played by Player I and Player II, as reasonable realizations of the combination of the notions of splitting and games.
%: splitting game, splitting* game, and splitting** game. 
Then, we consider two cardinal invariants for each game, so we define $3\times2=6$ numbers as follows: %$\sI, \sII, \sIstar, \sIIstar, \sIwstar, \sIIwstar$.  %Some of them are prove to be equal to the classical numbers: $\sI=\fraks_\sigma$ and $ \sII = \sIIstar = \sIIwstar =\frakc$ hold ($\fraks_\sigma$ denotes the $\sigma$-splitting number, see Definition \ref{defi:sigma_splitting}), so let us focus on the two remaining numbers $\sIstar$ and $\sIwstar$ %and give a rough explanation 
%in the rest of this section.

Fix a set $\scrA \subseteq \Pow(\omega)$.
	We call the game indicated by the following Table \ref{table:splittinggame} the \textit{splitting game} with respect to $\scrA$.
	
	\begin{table}[H]
		\caption{The splitting game}
		\centering
		\begin{tabular}{l|llllll}\label{table:splittinggame}
			Player I  & $n_0$ &       & $n_1$ &       & $\dots$ &     \\  \hline
			Player II &       & $i_0$ &       & $i_1$ &         & $\dots$
		\end{tabular}
	\end{table}
	
	Here, $n_0 < n_1 < n_2 < \dots < n_k < \dots$ are increasing numbers in $\omega$, $i_k$ ($k \in \omega$) are elements in $2$. %$\{0,1\}$.
	Player II wins when Player II played each of $0$ and $1$ infinitely often and there is $A \in \scrA$ such that 
	\begin{align}
		\{ n_k : k \in \omega \} \cap A = \{ n_k : k \in \omega \text{ and } i_k = 1\}. \label{eq:splittinggame} \tag{$*$}
	\end{align}

	Fix a set $\scrA \subseteq \Pow(\omega)$.
	We call the game indicated by the following Table \ref{table:splittingstargame} the \textit{splitting* game} with respect to $\scrA$.
	
	\begin{table}[H]
		\caption{The splitting* game}
		\centering
		\begin{tabular}{l|llllll}\label{table:splittingstargame}
			Player I  & $i_0$ &       & $i_1$ &       & $\dots$ &     \\  \hline
			Player II &       & $j_0$ &       & $j_1$ &         & $\dots$
		\end{tabular}
	\end{table}
	
	Here, $\seq{i_k : k  \in \omega}$ and $\seq{j_k : k  \in \omega}$ are in $2^\omega$. %sequences of elements in $2$. %$\{0,1\}$.
	Player II wins when either Player I did not play $1$ infinitely often or
	\begin{align*}
		\{ k \in \omega : j_k = 1 \} \in \scrA \text{ and  splits } \{ k \in \omega : i_k = 1\}. 
	\end{align*}

        In the splitting* game, when both $\{ k \in \omega : j_k = 1 \}$ and $\{ k \in \omega : i_k = 1 \}$ are finite, 
        the winner is defined as Player II, which is not necessarily obvious. Hence, we just define another game by switching the winner in this case. Namely,
        the \textit{splitting** game} with respect to $\scrA$ follows the same rule as the splitting* game,
        but only the winning condition of Player II is replaced by:
        Player II wins when \underline{either} Player I did not play $1$ infinitely often \textit{and Player II did}, \underline{or}
	\begin{align*}
		\{ k \in \omega : j_k = 1 \} \in \scrA \text{ and  splits } \{ k \in \omega : i_k = 1\}. 
	\end{align*}

	Excluding apparently trivial or undefined cases, we define the following six cardinal invariants using the corresponding games:
	 \begin{align*}
	 	\sI &= \min \{ \abs{\scrA} : \text{Player I has no winning strategy} \\
         & \hspace{15ex} \text{for the splitting game with respect to } \scrA \}, \text{ and} \\
	 	\sII &= \min \{ \abs{\scrA} : \text{Player II has a winning strategy} \\
         & \hspace{15ex} \text{for the splitting game with respect to } \scrA \}. \\
	 	\sIstar &= \min \{ \abs{\scrA} : \text{Player I has no winning strategy} \\
         & \hspace{15ex} \text{for the splitting* game with respect to } \scrA \}, \text{ and} \\
	 	\sIIstar &= \min \{ \abs{\scrA} : \text{Player II has a winning strategy} \\
         & \hspace{15ex} \text{for the splitting* game with respect to } \scrA \}.\\
         \sIwstar &= \min \{ \abs{\scrA} : \text{Player I has no winning strategy} \\
         & \hspace{15ex} \text{for the splitting** game with respect to } \scrA \}, \text{ and} \\
	 	\sIIwstar &= \min \{ \abs{\scrA} : \text{Player II has a winning strategy} \\
         & \hspace{15ex} \text{for the splitting** game with respect to } \scrA \}.
	 \end{align*}

     In Section \ref{sec:ZFC}, we show $\ZFC$ results on these numbers. In particular, we prove that some of them are equal to classical cardinal invariants: First, all the ``$\fraks^{\mathrm{II}}$-numbers'' are equal to the size of the continuum $\frakc$:
     \begin{teorema}[Theorem \ref{thm:sii}]
    $\sII=\sIIstar=\sIIwstar=\mathfrak{c}$ holds.
     \end{teorema}
     Also, $\sI$ is equal to the following classical cardinal invariant related to $\sigma$-splitting families defined as follows:
     \begin{defi}\label{defi:sigma_splitting}
        A family $\scrA \subseteq \ooo$ is a $\sigma$-splitting family if for any function $f \colon\omega\to [\omega]^\omega$, there is $x \in \scrA$ such that $x$ splits every $f(n)$. The $\sigma$-splitting number $\fraks_\sigma$ is the minimum size of a $\sigma$-splitting family.
    \end{defi}
    \begin{teorema}[Theorem \ref{thm:sI_equals_fraks_sigma}]
        $\mathfrak{s}_\sigma=\sI$ holds.
    \end{teorema}
    We also prove other several $\ZFC$ results summarized in Fig. \ref{fig:cardinals}.

     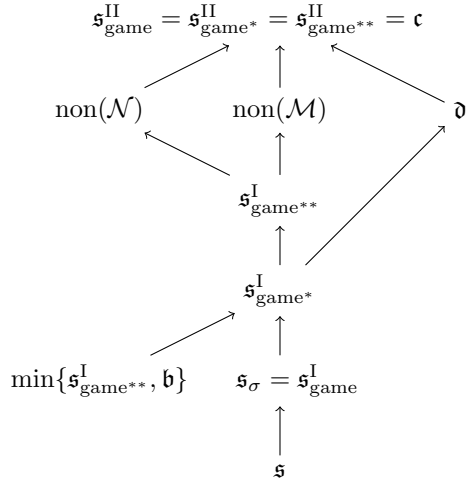
\begin{figure}[h]
   \centering
	\begin{tikzpicture}
		\tikzset{
			textnode/.style={text=black}, 
		}
		\tikzset{
			edge/.style={color=black, thin}, 
		}
		\tikzset{cross/.style={preaction={-,draw=white,line width=9pt}}}
		\newcommand{\w}{2.4}
		\newcommand{\h}{1.2}
		
		\node[textnode] (s) at (0,  0) {$\mathfrak{s}$};
		\node[textnode] (ssigma) at (0,  \h) {$~~~~\mathfrak{s}_\sigma=\sI$};
		\node[textnode] (sIstar) at (0,  2*\h) {$\sIstar$};
		\node[textnode] (sIwstar) at (0,  3*\h) {$\sIwstar$};
		\node[textnode] (nonm) at (0,  4*\h) {$\nonm$};
		\node[textnode] (nonn) at (-\w,  4*\h) {$\nonn$};
		\node[textnode] (d) at (\w,  4*\h) {$\frakd$};
		\node[textnode] (c) at (0,  5*\h) {$\sII=\sIIstar=\sIIwstar=\mathfrak{c}~~~~~$};
            \node[textnode] (min) at (-1*\w,  1*\h) {$\min\{\sIwstar,\frakb\}$};

	      \draw[->, edge] (s) to (ssigma);
	      \draw[->, edge] (ssigma) to (sIstar);
	      \draw[->, edge] (sIstar) to (sIwstar);
	      \draw[->, edge] (sIwstar) to (nonm);
	      \draw[->, edge] (sIwstar) to (nonn);
	      \draw[->, edge] (sIstar) to (d);
	      \draw[->, edge] (nonn) to (c);
	      \draw[->, edge] (nonm) to (c);
	      \draw[->, edge] (d) to (c);
	      \draw[->, edge] (min) to (sIstar);
	
	%\draw[double distance=2pt, edge] (addM) to (addE);

	\end{tikzpicture}
	\caption{A diagram of cardinal invariants related to games of splitting}\label{fig:cardinals}
\end{figure}
%Each arrow $\mathfrak{x}\to\mathfrak{y}$ denotes that the inequality $\mathfrak{x}\leq\mathfrak{y}$ holds in $\ZFC$

In Section \ref{sec:consistency}, we study the remaining two numbers $\sIstar$ and $\sIwstar$, whose values are not decided in Section \ref{sec:ZFC}. We show that these numbers can be different from $\fraks$ and $\fraks_\sigma$. Moreover, at the same time they can take different values from each other:
\begin{teorema}[Theorem \ref{thm:separation_of_three}, Remark \ref{rmk:preservingssigma}]
    $\fraks=\fraks_\sigma<\sIstar<\sIwstar$ consistently holds.
\end{teorema}
Recall that both kinds of games corresponding to $\sIstar$ and $\sIwstar$ share almost the same rule: Both players are \textit{basically} trying to construct an infinite subset of $\omega$ and Player II wins if Player I's infinite set is split by Player II's. If exactly one player breaks this basic rule, that is, the player plays a finite set, the winner is the other one. The difference of the two games only appears when both players play finite sets %(see Fig. \ref{table:winner}).
(see Table \ref{table:winner*}, \ref{table:winner**}).

\begin{table}[h]
	\begin{minipage}[h]{0.49\textwidth}
		\caption{The winner of splitting* games}\label{table:winner*}
            \vspace{-3pt}
		\centering
		\begin{tabular}{c|c|c}
        \diagbox{II}{I} & Infinite &  Finite   \\  
			\hline
			Infinite &      (Split or not) & II \\
			\hline
			Finite &      I & \textcolor{black}{II} 
		\end{tabular}
		
	\end{minipage}
	\begin{minipage}[h]{0.49\textwidth}
		\caption{The winner of splitting** games}\label{table:winner**}
            \vspace{-3pt}
		\centering
		\begin{tabular}{c|c|c}
			\diagbox{II}{I} & Infinite &  Finite   \\  
			\hline
			Infinite &      (Split or not) & II \\
			\hline
			Finite &      I & \textcolor{black}{I}
		\end{tabular}

	\end{minipage}
    %\caption{Two tables describing the winners of splitting* games and splitting** games, respectively }\label{table:winner}
\end{table}
%In both games, if both players play infinite sets, the winner is judged based on whether Player I's infinite set is split by Player II's. If Player I plays an infinite set but Player II plays a finite set, the winner is always Player I, and vice versa. If both players play finite sets, in the splitting* game the winner is always Player II, but in the splitting** game the winner is always Player I

However, the theorem above shows that this slight difference is crucial in the sense that the corresponding cardinal invariants $\sIstar$ and $\sIwstar$ can take distinct values.
To enjoy this result more, let us consider the following situation: Assume $\sIstar<\sIwstar$ and $\scrA\subseteq \mathcal{P}(\omega)$ is a witness of $\sIstar$. Since $|\scrA|=\sIstar<\sIwstar$, there is a winning strategy $\sigma$ for the \textit{splitting**} game with respect to $\scrA$. At the same time, $\scrA$ witnesses $\sIstar$, so there should be a play $y^*\in \scrA$ of II winning against $\sigma$ in the \textit{splitting*} game. However, whenever II follows the basic rule of this game, i.e., plays $1$ infinitely many times, 
the rules of the two games are the same, so II cannot find a winning play in $\scrA$ against $\sigma$. Therefore, the winning play $y^*$ cannot follow this basic rule. That is, if you($=$ Player II) want to win this splitting* game in this situation ($F$ and $\sigma$ are already fixed and you are choosing a play $y\in \scrA$), you should \textit{intentionally} break the basic rule, namely, play only $0$ after some point. Then, the opponent, who follows the strategy $\sigma$ winning against all the \textit{normal} plays in $\scrA$, will somehow end up playing only $0$ eventually and breaking the rule as well, and finally you will be judged as the winner, in this current rule set, the splitting* game. 

We also consider the background of this consistency result from the perspective of corresponding forcing posets and at the end of Section \ref{sec:consistency} we obtain a reasonable explanation to some extent of why this separation $\sIstar<\sIwstar$ consistently holds.
%from this point of view.

    In the rest of this section, we fix our notation.

    $(\forall^\infty n)$ and $(\exists^\infty n)$ are abbreviations to say ``for all but finitely many $n$" and ``there exist infinitely many $n$", respectively.

    For $m\leq\omega$ and $j<2$, $\seq{j}^m$ denotes the sequence of length $m$ whose values are all $j$. Namely, $\seq{j}^m=m\times\{j\}$.

    $\zero$ and $\one$ denote the set of all eventually $0$ and $1$ sequences of length $\omega$, respectively.

    %$\frakc$ denotes the cardinality of the continuum.
    $\nul$ and $\meager$ denote the Lebesgue null ideal and Baire first category ideal, respectively.
    As for the definitions of classical cardinal invariants (such as $\fraks$, $\frakb$, $\frakd$, $\mathfrak{p}$, $\non(\nul)$, $\non(\meager)$ and $\add(\nul)$), we refer the reader to \cite{blass2010combinatorial}.

	We also recall the notion of interval partitions. Let $\mathsf{IP}$ be the set of all interval partitions of $\omega$.
	For $\bar{I}=\langle I_n:n<\omega\rangle$ and $ \bar{J}=\langle J_m:m<\omega\rangle \in \mathsf{IP}$, we define
	\[
	\bar{I} \leq^* \bar{J} :\Leftrightarrow  (\forall^\infty m)(\exists n)(I_n \subseteq J_m).
	\]

    For a relational system $\mathbf{R} = \langle X, Y, R\rangle$, we use the following standard notation:
    \begin{align*}
    \frakb(\mathbf{R}) &= \min \{ \abs{A} : A \subseteq X \text{ an $\mathbf{R}$-unbounded family}\}, \\
    \frakd(\mathbf{R}) &= \min \{ \abs{B} : B \subseteq Y \text{ an $\mathbf{R}$-dominating family}\}.
    \end{align*}
    
    \section{ZFC results}\label{sec:ZFC}
 
    %(In this section, we consider games related to splitting families.
    %Moreover, using such games, we find a new cardinal invariant $\sIstar$ that differs from previously studied cardinal invariants related to $\mathfrak{s}$. It is worth mentioning that other variants of this cardinal have been used to prove important theorems in set theory or set-theoretic topology. An example of this can be found in \cite{splittingseparability}, where it is proved that the cardinal $\mathfrak{s}_{\omega,\omega}=\mathfrak{s}$. This is put into good use by showing that under $\mathfrak{s}\leq \mathfrak{a}$ there is a completely separable MAD family.)

    %In this section, we give the precise definitions of all games and corresponding cardinal invariants. 

    In this section, we prove all the relationships described in Figure \ref{fig:cardinals}.
    First, for Player II, the splitting** game is the hardest, the splitting* game is next and the splitting game is the easiest. More precisely,

    \begin{lem}\label{lem:spandspstar}
        \begin{enumerate}
            \item If Player II has a winning strategy for the splitting**[splitting*] game with respect to $\scrA$, then Player II has a winning strategy for the splitting*[splitting] game with respect to $\scrA$, respectively.
            \item If Player I has a winning strategy for the splitting[splitting*] game with respect to $\scrA$, then Player I has a winning strategy for the splitting*[splitting**] game with respect to $\scrA$, respectively.
        \end{enumerate}
    \end{lem}

    We omit the proof of this lemma. Moreover, we can easily show that all the numbers defined in Section \ref{sec:intro} are well-defined and hence $\sI \le \sIstar \le \sIwstar$ and $\sII \le \sIIstar \le \sIIwstar$ hold by the lemma.

    %$\sI$ is equal to the classical cardinal invariant $\fraks_\sigma$ defined as follows:

	\begin{thm}\label{thm:sI_equals_fraks_sigma}
		$\sI = \fraks_\sigma$ holds.
	\end{thm}
	\begin{proof}
		First we prove $\sI \le \fraks_\sigma$.
		Fix a $\sigma$-splitting family $\mathcal{A} \subseteq [\omega]^\omega$.
		We want to show that Player I has no winning strategy for the splitting game with respect to $\mathcal{A}$.
		Fix a strategy $\sigma \colon 2^{<\omega} \to \omega$ of Player I.
		Since $\zero \cup \one$ is a countable set and $\mathcal{A}$ is a $\sigma$-splitting family, we can take $A \in \mathcal{A}$ such that $A$ splits all $\{ \sigma(\bar{i} \upharpoonright k) : k \in \omega \}$ for $\bar{i} \in \zero \cup \one$.
		
		We consider the following $\bar{i} \in 2^\omega$:
		\[
		i_k = \begin{cases} 1 & \text{(if } \sigma(\bar{i} \upharpoonright k) \in A \text{)} \\ 0 & \text{(otherwise)} \end{cases}.
		\]
		If $\bar{i} \in \zero \cup \one$, then $A$ splits $\{ \sigma(\bar{i} \upharpoonright k) : k \in \omega \}$ by the choice of $A$. But by the choice of $\bar{i}$, this means $\bar{i} \not \in \zero \cup \one$, which is a contradiction.
		So we have $\bar{i} \not \in \zero \cup \one$.
		This observation and the choice $\bar{i}$ imply $\bar{i}$ is a winning play of Player II against the strategy $\sigma$.
		So we have proved Player I has no winning strategy for the splitting game with respect to $\mathcal{A}$.

		Next, we prove $\fraks_\sigma \le \sI$.
		Fix a family $\mathcal{A} \subseteq \Pow(\omega)$ such that Player I has no winning strategy for the splitting game with respect to $\mathcal{A}$.
		We want to show that $\mathcal{A}$ is a $\sigma$-splitting family.
		Take $f \colon \omega \to [\omega]^\omega$.
		We shall find an $A \in \scrA$ such that $A$ splits $f(n)$ for every $n \in \omega$.
		Take $f' \colon \omega \to [\omega]^\omega$ such that $\range(f) = \range(f')$ and each element of $\range(f)$ appears in the range of $f'$ infinitely often.
		For $m, n \in \omega$, we let $f'(n)(m)$ denote the $m$-th element of $f'(n)$ in ascending order.
		
		Consider the following strategy $\sigma$ of Player I.
		First $\sigma$ plays $f'(0)(0)$. 

        From then on, $\sigma$ will play the elements of $f'(0)$ in turn until Player II changes the value of play. After that, $\sigma$ plays $f'(1)(k)$ next. Here $k$ is the smallest number such that $f'(1)(k)$ exceeds the natural number that $\sigma$ has played so far. Continue this process.
		
		Table \ref{table:situation} shows an example.
		
		\begin{table}[H]
    		\caption{An example of the situation}
			\centering
			\begin{tabular}{l|llllllllll}\label{table:situation}
				Player I  & $f'(0)(0)$ & & $f'(0)(1)$ & & $f'(0)(2)$ & & $f'(1)(k)$ & $\dots$ &     \\  \hline
				Player II & & $0$ & & $0$ & & 1 &  & $\dots$
			\end{tabular}
		\end{table}
		
		Since this $\sigma$ is not a winning strategy, there is $A \in \scrA$ and $\bar{i} \in 2^\omega \setminus (\zero \cup \one)$ such that the equation (\ref{eq:splittinggame}) holds for $n_k = \sigma(\bar{i} \upharpoonright k)$.
		This implies $A$ splits all elements in $\range(f)$ by the definition of $\sigma$.
	\end{proof}
	
	\begin{thm}\label{thm:sii}
		$\sII = \frakc$ holds.
	\end{thm}
	\begin{proof}
		Fix $\scrA \subseteq \Pow(\omega)$ such that Player II has a winning strategy for the splitting game with respect to $\scrA$. We shall show that $\scrA$ is of size $\frakc$.
		Consider the game tree $T \subseteq \omega^{<\omega}$ that the winning strategy determines.
		
		First, assume the following.
		\begin{itemize}
			\item (Case 1) There is an even number $k \in \omega$ and there is a  $\sigma \in \omega^k \cap T$ such that for every $m > \sigma(k-2)$ there is $i_m < 2$ such that for every $\tau \in T$ extending $\sigma$ and every $r \in [\abs{\sigma}, \abs{\tau})$, $\tau(r) = m$ implies $\tau(r+1) = i_m$.
		\end{itemize}
		
		Fix the witness $k, \sigma, \seq{i_m : m > \sigma(k-2)}$ for Case 1.
		Take an infinite set $A \subseteq [\sigma(k-2), \omega)$ and $i^* < 2$ such that $i_m = i^*$ for every $m \in A$.
		Enumerate $A$ in ascending order as $A = \{ a_i : i \in \omega \}$.

        Then considering the play of Player I that play $a_0, a_1, a_2, \dots$ in turn after $\sigma$, Player II that obeys the winning strategy plays $i^*$ eventually. So Player II loses, which is a contradiction.
		%Consider the next play.
		%\begin{table}[H]
		%	\centering
		%	\begin{tabular}{l|llllllllllll}
		%		Player I  & $\sigma(0)$ &             & $\sigma(2)$ &             & $\dots$ & $\sigma(k-2)$ &               & $a_0$ &       & $a_1$ &       & $\dots$ \\ \hline
		%		Player II &             & $\sigma(1)$ &             & $\sigma(3)$ &         &               & $\sigma(k-1)$ &       & $i^*$ &       & $i^*$ &        
		%	\end{tabular}
		%\end{table}
		%
		%Then Player II eventually plays $i^*$. So Player II loses. This is a contradiction.
		
		So Case 1 is false.
		Thus we have
		\begin{itemize}
			\item (Case 2) For every even number $k \in \omega$ and every $\sigma \in \omega^k \cap T$, there is $m > \sigma(k-2)$ such that for every $i < 2$, there is $\tau \in T$ extending $\sigma$ and there is $r \in [\abs{\sigma}, \abs{\tau})$ such that $\tau(r) = m$ and  $\tau(r+1) = 1 - i$.
		\end{itemize}
		Then we can construct a perfect subtree of $T$ whose distinct paths yield distinct elements of $\scrA$.
		
        To do it, first we put $\sigma_\varnothing = \varnothing$.
		Suppose we have $\seq{\sigma_s : s \in 2^{\le l}}$.
		Then for each $s \in 2^l$, we can take $m_s > \sigma_s(\abs{\sigma_s} - 2)$ such that for every $i < 2$ we can take $\sigma_{s \append i}(\abs{\sigma_{s \append i}} - 2) = m_s$ and $\sigma_{s \append i}(\abs{\sigma_{s \append i}} - 1) = i$.
		
		Now for each $f \in 2^\omega$, we put $\sigma_f$ by $\sigma_f = \bigcup_{n \in \omega} \sigma_{f \upharpoonright n}$.
		
		For each $f \in 2^\omega$, we have $\sigma_f \in [T]$. So Player II wins at the play $\sigma_f$.
		So by the definition of the game, we can take $A_f \in \scrA$ such that
		\[
		\{ \sigma_f(r) : r \text{ even} \} \cap A_f = \{ \sigma_f(r) : r \text{ even and } \sigma_f(r+1) = 1 \}.
		\]
		
		We now claim that if $f$ and $g$ are distinct element of $2^\omega$, then we have $A_f \ne A_g$.
		Let $n := \min \{ n' : f(n') \ne g(n') \}$. Put $s = f \upharpoonright n = g \upharpoonright n$.
		We may assume that $f(n) = 0$ and $g(n) = 1$.
		We have the following equations:
		\begin{align*}
			&\sigma_{s \append 0}(\abs{\sigma_{s \append 0}}-2) = \sigma_{s \append 1}(\abs{\sigma_{s \append 1}}-2) = m_s, \\
			&\sigma_{s \append 0}(\abs{\sigma_{s \append 0}}-1) = 0, \text{ and } \sigma_{s \append 1}(\abs{\sigma_{s \append 1}}-1) = 1.
		\end{align*}
		So we have $m_s \not \in A_f$ and $m_s \in A_g$. Thus $A_f \ne A_g$.
		
		Therefore we have $\abs{\scrA} \ge \abs{\{A_f : f \in 2^\omega\}} = \frakc$.
	\end{proof}

    By the remark below Lemma \ref{lem:spandspstar}, %we have also $\fraks_\sigma \le \sIstar$ and $\sIIstar = \sIIwstar = \frakc$.
    we obtain $\sI=\fraks_\sigma \le \sIstar \le \sIwstar$ and $\sII = \sIIstar = \sIIwstar=\frakc$.
    Let us further investigate the remaining $\sIstar$ and $\sIwstar$.
Firstly, let us introduce some notation:
\begin{defi}
	Let $x,y\in 2^\omega$.
	\begin{enumerate}
		
		\item $x\rsp y$ if $x(n)=y(n)=1$ and $x(m)=1$ and $y(m)=0$ for infinitely many $n,m<\omega$, i.e., $x^{-1}(\{1\})$ and $y^{-1}(\{1\})$ are infinite and  $x^{-1}(\{1\})$ is split by $y^{-1}(\{1\})$. $y\rp x$ if $\lnot(x\rsp y)$. Note that $y\rp x$ holds whenever $y\in\zero$.  
		
		\item Let $j\in2$ and $n<\omega$.
		$y\rp_{j,n} x$ if for all $m\geq n$, $x(m)=0$ or $y(m)=1-j$ holds. 
		Note $\rp=\bigcup_{j\in2}\bigcup_{n<\omega}\rp_{j,n}$.
		
		\item $y\tri x$ if $x\in 2^\omega\setminus\zero$ and $y\rp x$, i.e., I wins with the play $x$ against the play $y$ of II in the splitting* game. $y\trii x$ if either $x\in 2^\omega\setminus\zero$ and $y\rp x$ or $y\in\zero$, i.e., I wins with the play $x$ against the play $y$ of II in the splitting** game.
		
		\item $\str$ denotes the set of all I's strategies,
		namely, $\str\coloneq2^{(2^{<\omega})}$. For $\sigma\in\str$, $\sigma*y$ denotes the play of I according to the strategy $\sigma$ and the play $y$ of II, namely, $\sigma*y(n)\coloneq\sigma(y\on n)$ for $n<\omega$. $y\rp \sigma$ denotes $y\rp \sigma*y$
        and analogously for $\tri$ and $\trii$.
		\item Define relational systems $\rs$ and $\rss$ by $\rs\coloneq\langle 2^\omega,\str,\tri\rangle$ and $\rss\coloneq\langle 2^\omega,\str,\trii\rangle$. Thus, $\sIstar=\frakb(\rs)$ and $\sIwstar=\frakb(\rss)$ hold.

	\end{enumerate}
\end{defi}

\begin{lem}
	\label{lem_s_basics}
	%Let $x,y\in2^\omega$.
	\begin{enumerate}
		%\item If $y\in \zero$, then $y\rp x$.
		\item $y\trii x\Leftrightarrow y\tri x$ or $y\in\zero$. %In particular, $\sIstar\leq\sIwstar$.
		\item If $y\tri x$ (or $y\trii x$), then $y\rp x$.
	\end{enumerate}
	
\end{lem}

When investigating $\sIstar$ and $\sIwstar$, it is sometimes useful to restrict strategies:   
\begin{defi}
	\begin{itemize}
		\item $\str_\infty\coloneq\{\sigma\in\str:\text{ for all }y\in\zero\cup\one, \sigma*y\in2^\omega\setminus\zero\}$.
		\item $\str_\infty^\one\coloneq\{\sigma\in\str:\text{ for all }y\in\one, \sigma*y\in2^\omega\setminus\zero\}$.
		\item $\rs_\infty\coloneq\langle 2^\omega,\str_\infty,\tri\rangle$.
		\item $\rss_\infty\coloneq\langle 2^\omega,\str_\infty^\one,\trii\rangle$.
	\end{itemize}

\end{defi}

Note that:
\begin{equation*}
    \text{for }y\in\zero\cup\one, ~y\tri\sigma \iff \sigma*y\in2^\omega\setminus\zero, \text{ and}
\end{equation*}

\begin{equation*}
    \text{for }y\in\one, ~y\trii\sigma\iff\sigma*y\in2^\omega\setminus\zero.
\end{equation*}

\begin{lem}
	$\frakb(\rs_\infty)=\frakb(\rs)=\sIstar$ and $\frakb(\rss_\infty)=\frakb(\rs)=\sIwstar$.
\end{lem}
\begin{proof}
	$\frakb(\rs_\infty)\leq\frakb(\rs)$ is clear. To show $\frakb(\rs)\leq\frakb(\rs_\infty)$, let $F\subseteq2^\omega$ of size $<\frakb(\rs)$. $\frakb(\rs)=\sIstar\geq\sI=\fraks_\sigma$ is infinite, so $F^\prime\coloneq F\cup\zero\cup\one$ has size $<\frakb(\rs)$ and hence some $\sigma\in\str$ wins all $y\in F^\prime$. This $\sigma$ has to be in $\str_\infty$.
	The latter equation is proved in the same way.
\end{proof}

\begin{lem}
	\label{lem_sIwstarleqnonm}
	For any $\sigma\in\str_\infty^\one$, $\{y\in 2^\omega:y\trii\sigma\}\in\mathcal{M}$. In particular, $\sIwstar=\frakb(\rss_\infty)\leq\nonm$ holds.
\end{lem}
\begin{proof}
For $\sigma\in\str_\infty^\one$, we have:
\begin{align*}
    \{y\in 2^\omega:y\trii\sigma\}&\subseteq\{y\in2^\omega:y\rp\sigma*y\} \\
    & =\bigcup_{j\in2}\bigcup_{n<\omega}\bigcap_{m\geq n}\left(\{y\in2^\omega:\sigma*y(m)=0\}\cup\{y\in2^\omega:y(m)=j\}\right)\in \mathcal{M},
\end{align*}
since $\sigma\in\str_\infty^\one$.
\end{proof}

    \begin{thm}
    \label{thm:sIstar_le_frakd}
		$\sIstar = \frakb(\rs_\infty) \le \frakd$ holds.
    \end{thm}
    \begin{proof}
            Let $\mathcal{I} \subseteq \mathsf{IP}$ be a dominating family of size $\frakd$ with respect to interval partitions. For $\bar{I} = \langle I_n : n < \omega \rangle \in \mathcal{I}$,  let $x_{\bar{I}}\in2^\omega$ be the characteristic function of $\bigcup_{n < \omega} I_{2n}$.
            %define $x_I = \bigcup_{n < \omega} I_{2n}$. (technically x_I should be a member of 2^\omega) 
            % and set $\scrA = \{ x_I : I \in \mathcal{I} \}$. (no need to set \scrA?)
        Take a strategy $\sigma \in \str_\infty$ of Player I. We show that there is an interval partition $\bar{I}\in \mathcal{I}$ such that $\neg (x_{\bar{I}} \tri \sigma)$.

        Since $\sigma \in \str_\infty$, for all $s \in 2^{<\omega}$ and $i < 2$, we can find $m_s^i > 0$ such that $\sigma(s \append \seq{i}^{m_s^i}) = 1$.
        Fix such $m_s^i$'s and define a sequence $\langle j_k : k< \omega \rangle$ of natural numbers as follows:
        \begin{align*}
            j_0 &= 0, \\
            j_{2k + 1} &= j_{2k} + \max \{ m_s^1 : s\in 2^{j_{2k}} \} + 1  \text{ , and}\\
            j_{2k + 2} &= j_{2k + 1} + \max \{ m_s^0 : s\in 2^{j_{2k + 1}} \} + 1.
        \end{align*}
        Let $J_{k} = [ j_{2k}, j_{2k + 2} )$. Since $\mathcal{I}$ is a dominating family, there is $\bar{I} = \langle I_n : n < \omega \rangle \in \mathcal{I}$ such that $\bar{J} \leq^* \bar{I}$, that is, 
        \[ (\exists n_0 < \omega) (\forall n \ge n_0) (\exists k) (J_k \subseteq I_n).   \]

        Let $I_n = [i_n, i_{n + 1})$. Take an arbitrary $n \ge n_0$ and $k$ such that $J_k \subseteq I_n$. Note that $j_{2k}, j_{2k + 1} \in J_k$. First, we consider the case $n$ is even. Set $m = m_{x_{\bar{I}}\upharpoonright j_{2k}}^1 > 0$. By the construction of $j_{2k + 1}$, we have $i_n \le j_{2k} + m < j_{2k + 1}\le i_{n + 1}$. So it holds that $\sigma*x_{\bar{I}}(j_{2k} + m) = \sigma((x_{\bar{I}}\upharpoonright j_{2k}) \append \seq{1}^{m}) = 1$ and $x_{\bar{I}}(j_{2k} + m) = 1$. 

        In the case $n$ is odd, set $m = m_{x_{\bar{I}}\upharpoonright j_{2k + 1}}^0$. Then, by a similar argument, we have that $i_n \le j_{2k + 1} + m < i_{n + 1}$, $\sigma*x_{\bar{I}}(j_{2k + 1} + m) = 1$, and $x_{\bar{I}}(j_{2k + 1} + m) = 0$. Therefore, $\neg (x_{\bar{I}} \tri \sigma)$.
    \end{proof}

	\begin{thm}\label{thm:sIstarleqnonN}
		For any $\sigma\in\str$, $\{ y \in 2^\omega : y\tri \sigma \} \in \mathcal{N} $. In particular, $\sIstar \le \non(\nul)$ holds.
	\end{thm}

	To prove this theorem, we prepare some lemmas.
	
	\begin{lem}
		Let $I = [i, j)$ be an interval in $\omega$.
		Let $s \in \{0, 1\}^i$, $\sigma \colon \{0, 1\}^{<j} \to 2$, and $\varepsilon \in 2$.
		Set
		\begin{align*}
		B^I_{s,\varepsilon}(\sigma) = \{ x \in \{0, 1\}^j : &s \subseteq x, (\exists k \in I)(\sigma(x \upharpoonright k) = 1) \text{, and } \\
		& (\forall k \in I)(\sigma(x \upharpoonright k) = 1 \rightarrow x(k) = \varepsilon)\}.
		\end{align*}
		Then we have
		\[
		\frac{\abs{B^I_{s,\varepsilon}(\sigma)}}{2^{j-i}} \le \frac12.
		\]
	\end{lem}
	\begin{proof}
		Induction on $\abs{I}$.
		If $\abs{I} = 1$ then $\abs{B^I_{s,\varepsilon}(\sigma)} = \abs{\{ s \append \seq{\varepsilon} \}} = 1$. So in this case, the lemma is proven.
		Suppose $\abs{I} \ge 2$.
		If $\sigma(s) = 1$, then $\abs{B^I_{s,\varepsilon}(\sigma)} \le \abs{\{ x : s \append \seq{\varepsilon} \subseteq x \}} = 2^{j-i-1}$.
		Otherwise, by the induction hypothesis, we have
		\[
		\abs{B^I_{s,\varepsilon}(\sigma)} = \abs{B^{[i+1, j)}_{s \append \seq{0},\varepsilon}(\sigma)} + \abs{B^{[i+1, j)}_{s \append \seq{1},\varepsilon}(\sigma)} \le 2^{j-(i+1)-1} + 2^{j-(i+1)-1} = 2^{j-i-1}.
		\]
	\end{proof}

	\begin{lem}\label{lem:halveiteration}
		Let $a < b < \omega$.
		Let $\bar{I} = \seq{I_n : a \le n < b}$ be a sequence of consecutive intervals in $\omega$ and put $m := \min I_a$ and $M := \max I_{b-1} + 1$.
		Let $\sigma \colon \{0, 1\}^{<M} \to 2$ and $\varepsilon \in 2$.
		Set
		\begin{align*}
			B^{\bar{I}}_{\varepsilon}(\sigma) = \{ x \in \{0, 1\}^M : &(\forall n \in [a, b))[(\exists k \in I_n)(\sigma(x \upharpoonright k) = 1) \text{, and } \\
			& \hspace{13ex}(\forall k \in I_n)(\sigma(x \upharpoonright k) = 1 \rightarrow x(k) = \varepsilon)]\}.
		\end{align*}
		Then we have
		\[
		\frac{\abs{B^{\bar{I}}_{\varepsilon}(\sigma)}}{2^M} \le \frac1{2^{b-a}}.
		\]
	\end{lem}
	\begin{proof}
		Use the previous lemma and induct on $b-a$.
	\end{proof}

	We use the following theorem due to Goldstern.

	\begin{fact}[{\cite{Goldstern1993}}]
		Let $A \subseteq \mathsf{IP} \times 2^\omega$ be a   $\boldsymbol{\Sigma}^1_1$ set.
		Suppose that the vertical section $A_{\bar{I}}$ is null for every $\bar{I} \in \mathsf{IP}$ and $A_{\bar{I}} \subseteq A_{\bar{J}}$ for every $\bar{I}, \bar{J} \in \mathsf{IP}$ with $\bar{I} \leq^* \bar{J}$.
		Then $\bigcup_{\bar{I} \in \mathsf{IP}} A_{\bar{I}}$ is null.
	\end{fact}

	Goldstern proved this theorem not with $\mathsf{IP}$, but with $\omega^\omega$. But these two versions can easily be shown to be equivalent.
	
	\begin{proof}[Proof of Theorem \ref{thm:sIstarleqnonN}]
		%Let $\mathcal{A} \subseteq \Pow(\omega)$ be a non-null set of size $\non(\nul)$.
		%We will show that Player I has no winning strategy for the splitting* game with respect to $\mathcal{A}$.
		Fix a strategy $\sigma \colon 2^{<\omega} \to 2$ of Player I.
		
		For $\bar{I} \in \mathsf{IP}$ and $\varepsilon \in 2$, define
		\[
		C^{\bar{I}}_\varepsilon =  \bigcup_{a \in \omega} \bigcap_{b > a} B^{\bar{I} \upharpoonright [a, b)}_\varepsilon(\sigma).
		\]
		By Lemma \ref{lem:halveiteration}, this set $C^{\bar{I}}_\varepsilon$ is null.
		
		Moreover when $\bar{I} \leq^* \bar{J}$, we have $C^{\bar{I}}_\varepsilon \subseteq C^{\bar{J}}_\varepsilon$. Also, the set $\{ (\bar{I}, x) : x \in C^{\bar{I}}_\varepsilon \}$ is clearly a Borel set. Therefore, we can apply Goldstern's theorem to get that $\bigcup_{\bar{I} \in \mathsf{IP}} C^{\bar{I}}_\varepsilon$ is null.
		
		Moreover, we can easily observe that
		\[
		\{ y \in 2^\omega : y\tri \sigma \} \subseteq \bigcup_{\bar{I} \in \mathsf{IP}} C^{\bar{I}}_0 \cup \bigcup_{\bar{I} \in \mathsf{IP}} C^{\bar{I}}_1 \in \mathcal{N}. \qedhere
		\]
        \end{proof}
  %\[
%		\{ x \in 2^\omega : \text{the strategy $\sigma$ wins the play $x$ in splitting* game} \} \subseteq \bigcup_{\bar{I} \in \mathsf{IP}} C^{\bar{I}}_0 \cup \bigcup_{\bar{I} \in \mathsf{IP}} C^{\bar{I}}_1.
%		\]
%		So we can take $x \in \mathcal{A}$ that avoids this set. This means $\sigma$ is not winning strategy for the splitting* game with respect to $\mathcal{A}$.

\begin{cor}
	\label{lem_s1game**_ineqs}
	For any $\sigma\in\str$, $\{ y \in 2^\omega : y\trii \sigma \} \in \mathcal{N} $. In particular, $\sIwstar \le \non(\nul)$ holds.
\end{cor}
\begin{proof}
	By Lemma \ref{lem_s_basics} and Theorem \ref{thm:sIstarleqnonN}, $\{y\in 2^\omega:y\trii\sigma\}=\{y\in 2^\omega\setminus\zero:y\tri\sigma\}\cup\zero$ is null.
\end{proof}

\begin{lem}
\label{lem:sIstar_lowerbound_min}
	$\sIstar\geq\min\{\sIwstar,\frakb\}$.
\end{lem}
\begin{proof}
	Let $F\subseteq2^\omega$ of size $<\min\{\sIwstar,\frakb\}$ and $F^-\coloneq F\setminus\zero$.
	For $x\in F^-$ and $k<\omega$, let $i_k$ denote the $k$-th element of $x^{-1}(\{1\})\in\ooo$, $I^x_k\coloneq\left[i_{k-1},i_k\right)$ ($i_{-1}\coloneq0)$.
	Define an interval partition $I^x$ by $I^x\coloneq\langle I^x_k:k<\omega\rangle$.
	Since $|F^-|<\frakb$, there is an interval partition $J=\langle J_m:m<\omega\rangle$ dominating all $\{I^x:x\in F^-\}$,
	namely, for any $x\in F^- \text{ and for all but finitely many }m<\omega,
	\text{ there is }k<\omega\text{ such that }I^x_k\subseteq J_m$.
	Particularly we have:
	%	\begin{equation}
		%		\text{for any }x\in F^- \text{ and for all but finitely many }m<\omega,\\
		%		\text{ there is }n\in J_m\text{ such that }x(n)=1.
		%	\end{equation}
	\begin{equation}
		\label{eq_F_minus}
		(\forall x\in F^-)(\forall^\infty m<\omega)
		(\exists i\in J_m)(x(i)=1).
	\end{equation}
	Since $|F^-|<\sIwstar$, there is a strategy $\sigma\in\str$ such that $x\trii\sigma$ for all $x\in F^-$.
	%	Pick some $A\coloneq\{m_0<m_1<\cdots\}\in\oo$ and $D^\prime\coloneq\{d_0<d_1<\cdots\}\in[D]^\omega$ such that $\max J_{m_i}<d_i<\min J_{m_{i+1}}$ holds for all $i<\omega$.
	Define the strategy $\sigma^\prime$ by, for $n<\omega$, $s\in2^n$ and $m<\omega$ with $n\in J_m$:
	\begin{equation}
		\label{eq_sigmaPrime}
		\sigma^\prime(s)\coloneq
		\begin{cases}
			\sigma(s) & \text{if } s(i)=1\text{ for some }i\in J_{m-1}~(\text{ put }J_{-1}\coloneq\emptyset),\\
			1 & \text{otherwise. }
		\end{cases}
	\end{equation}
	By \eqref{eq_F_minus}, every $x\in F^-$ satisfies $\sigma^\prime(x\on n)=\sigma(x\on n)$ for all but finitely many $n$, and every $x\in\zero$ satisfies $\sigma^\prime(x\on n)=1$ for all but finitely many $n$ by definition.
	Thus, $x\tri \sigma^\prime$ holds for any $x\in F\subseteq F^-\cup\zero$.	
\end{proof}

%Figure \ref{fig:cardinals} illustrates the $\ZFC$-provable relationship we have shown (except for Lemma \ref{lem:sIstar_lowerbound_min}):

\section{Consistency results}\label{sec:consistency}

%Actually, every arrow $\mathfrak{x}\to\mathfrak{y}$ in Figure \ref{fig:cardinals} is proved to be strict, i.e., $\mathfrak{x}<\mathfrak{y}$ consistently holds, except for $\mathfrak{s}\to\mathfrak{s}_\sigma$. 

When showing the consistency of $\mathfrak{s}=\mathfrak{s}_\sigma<\sIstar<\sIwstar$, we will use the \textit{Fr-limit} method to keep $\sIstar$ small in the forcing iteration, %This method stems from Miller's work \cite{Mil81} and 
formalized by Mej{\'\i}a in \cite{mejia2019matrix}. However, to prove Main Lemma \ref{mainlem} we need a more detailed formulation, so let us introduce the general theory of Fr-limits in detail in this section.
 
 \begin{defi}
	Let $\P$ be a poset.
	\begin{enumerate}
		\item For a countable sequence $\bar{p}=\langle p_m:m<\omega\rangle\in\P^\omega$, define $\Vdash_\P\dot{W}(\bar{p})\coloneq\{m<\omega:p_m\in\dot{G}\}$ where $\dot{G}$ denotes the canonical $\P$-name of a generic filter.
		\item $Q\subseteq \P$ is Fr-linked if there exists a function $\lim\colon Q^\omega\to\P$ such that for any countable sequence 
		$\bar{q}\in Q^\omega$, 
		\begin{equation}
			\textstyle{\lim\bar{q}} \Vdash |\dot{W}(\bar{q})|=\omega.
		\end{equation}
		
		\item For an infinite cardinal $\mu$, $\P$ is $\mu$-Fr-linked if it is a union of $\mu$-many Fr-linked components. When $\mu=\aleph_0$, we use $\sigma$ instead as usual. Define $<\mu$-Fr-linkedness in the same way (for uncountable $\mu$).

	\end{enumerate}

	We often say ``$\P$ has Fr-limits'' instead of ``$\P$ is $\sigma$-Fr-linked''.

\end{defi}
Since any singleton $\{q\}\subseteq\P$ is Fr-linked, we have:
\begin{lem}
\label{lem:poset_is_its_size_Frlinked}
    Every poset $\P$ is $|\P|$-Fr-linked. In particular, Cohen forcing $\mathbb{C}$ has Fr-limits.
\end{lem}

We formulate a finite support iteration (fsi) of $<\kappa$-Fr-linked forcings:
%(In the definition below, ``simple'' means we do not use $\P^-_\xi$. )
\begin{defi}
	\label{dfn_Gamma_iteration}
		\begin{itemize}
		\item 
A fsi $\P_\gamma=\langle(\P_\xi,\qd_\xi):\xi<\gamma\rangle $ of ccc forcings is a $<\kappa$-Fr-iteration with witnesses $\langle\theta_\xi:\xi<\gamma\rangle$ and $\langle\dot{Q}_{\xi,\zeta}:\zeta<\theta_\xi,\xi<\gamma\rangle$ if for any $\xi<\gamma$, $\theta_\xi$ is a cardinal $<\kappa$ and
			%\item $\qd_\xi$ is 
			%\label{item_Q_is_P_minus_name}
			%$\qd_\xi$ and 
			$\langle\dot{Q}_{\xi,\zeta}:\zeta<\theta_\xi\rangle$ are $\P_\xi$-names satisfying: %\in\Lambda_\mathrm{Fr}(\qd_\xi)
			\[
			\Vdash_{\P_\xi}\dot{Q}_{\xi,\zeta}\subseteq\qd_\xi\text{ is Fr-linked for }\zeta<\theta_\xi\text{ and }\bigcup_{\zeta<\theta_\xi}\dot{Q}_{\xi,\zeta}=\qd_\xi.
			\]
		\item $p\in\P_\gamma$ is determined if for each $\xi\in\dom(p)$, there is $\zeta_\xi<\theta_\xi$ such that  $\Vdash_\xi p(\xi)\in \dot{Q}_{\xi,\zeta_\xi}$. Note that there are densely many determined conditions (proved by induction on $\gamma$).

	\end{itemize}
\end{defi}

\begin{defi}
	\label{dfn_uniform_delta_system}
	Let $\delta$ be an ordinal and 
	%$h\in H$ and 
	$\bar{p}=\langle p_m: m<\delta\rangle\in(\P_\gamma)^\delta$.
	$\bar{p}$ is a uniform $\Delta$-system if:
	\begin{enumerate}
		%\item There is $h\in H$ such that $p_ m\in\P_\gamma^h$ for all $ m<\delta$.
		\item Each $p_m$ is determined as witnessed by $\langle\zeta^m_{\xi} :\xi\in\dom(p_m)\rangle$.
		\item $\{\dom(p_m): m<\delta\}$ forms a $\Delta$-system with some root $\nabla$.
		\item For $\xi\in\nabla$, all $\zeta^m_\xi$ are the same, i.e., all $p_m(\xi)$ are forced to be in a common Fr-linked component.
	
		\item All $|\dom(p_ m)|$ are the same $n^\prime$ and $\dom(p_ m)=\{\a_{n, m}:n<n^\prime\}$ is the increasing enumeration.
		\item There is $r^\prime\subseteq n^\prime$ such that $n\in r^\prime\Leftrightarrow\a_{n, m}\in\nabla$ for $n<n^\prime$.
		\item For $n\in n^\prime\setminus r^\prime$, $\langle\a_{n, m}: m<\delta \rangle$ is (strictly) increasing.

	\end{enumerate}
\end{defi}

\begin{defi}
	\label{dfn_of_lim_ite}
	Let $\bar{p}=\langle p_m: m<\omega\rangle\in(\P_\gamma)^\omega$ be a uniform $\Delta$-system with root $\nabla$.
	We (inductively) define $p^\infty=\lim\bar{p}$ as follows:
	\begin{enumerate}
		\item $\dom(p^\infty)\coloneq\nabla$.
		\item $p^\infty\on\xi\Vdash_\xi p^\infty(\xi)\coloneq\lim\langle p_m(\xi):m\in\dot{W}(\bar{p}\on\xi)\rangle$ for $\xi\in\nabla$, where $\bar{p}\on\xi\coloneq\langle p_m\on\xi:m<\omega\rangle\in(\P_\xi)^\omega$. \label{eq_lim_second}
	\end{enumerate}
\end{defi}
To see that the second item is a valid definition, $\dot{W}(\bar{p}\on\xi)$ has to be infinite and this is inductively shown as follows:

%By (inductively) proving the following lemma we verify Definition \ref{dfn_of_lim_ite} above:

\begin{lem}[\cite{mejia2019matrix}]
	Let $\bar{p}=\langle p_m: m<\omega\rangle\in(\P_\gamma)^\omega$ be a uniform $\Delta$-system and $p^\infty\coloneq\lim\bar{p}$.
	Then,
	%$p^\infty\Vdash_\gamma\exists^\infty m<\omega~p_m\in\dot{G}$.
	$p^\infty\Vdash_\gamma|\dot{W}(\bar{p})|=\omega$.
\end{lem}

\begin{proof}
	Induction on $\xi$.
	
	\textit{Successor Step}. Let $\bar{r}=\langle r_m=(p_m,\dot{q}_m): m<\omega\rangle\in(\P_{\xi+1})^\omega$ be a uniform $\Delta$-system with root $\nabla$ and we may assume that $\xi\in\nabla$. Put $\bar{p}\coloneq\langle p_m:m<\omega\rangle$. %and $\bar{q}\coloneq\langle \dot{q}_m:m<\omega\rangle$.
	By Induction Hypothesis $p^\infty\coloneq\lim\bar{p}\Vdash_{\P_\xi}\dot{A}\coloneq\dot{W}(\bar{p})\in\ooo$. %\{m<\omega:p_m\in\dot{G}_\xi\}
	Since all $\dot{q}_m$ are forced to be in a common Fr-linked component,
	$p^\infty\Vdash_{\P_\xi} \dot{q}^\infty\coloneq\lim\langle \dot{q}_m:m\in\dot{A}\rangle$ is a valid definition and $p^\infty\Vdash_{\P_\xi} \dot{q}^\infty\Vdash_{\qd_\xi} \dot{B}\coloneq\{m\in\dot{A}:\dot{q}_m\in\dot{H}_\xi\}\in[\dot{A}]^{\omega}$ where $\dot{H}_\xi$ denotes the canonical name of the $\qd_\xi$-generic filter.
	Thus, $\lim\bar{r}=(p^\infty,\dot{q}^\infty)$ forces %\{m<\omega:r_m\in\dot{G}_{\xi+1}\}
	$\dot{W}(\bar{r})\supseteq\{m<\omega:(p_m,\dot{q}_m)\in \dot{G}_\xi*\dot{H}_\xi\}\supseteq\dot{B}\in\ooo$.
	
	\textit{Limit Step}. 
	Let $\gamma$ be a limit, $\bar{p}=\langle p_m: m<\omega\rangle\in(\P_\gamma)^\omega$ a uniform $\Delta$-system with root $\nabla$. Assume $m_0<\omega$ and $q\leq p^\infty\coloneq\lim\bar{p}$ and we will find $m>m_0$ and $r\leq q$ forcing $p_m\in\dot{G}_\gamma$.
	  %\left[\max{r},n^\prime\right)$ and note that all $\langle \a_{n,m}:m<\omega\rangle$ are increasing for $n\in I$.
	To avoid triviality, we may assume that $\cfi(\gamma)=\omega$ and $I\coloneq\{n<n^\prime:\sup_{m<\omega}\a_{n,m}=\gamma\}\neq\emptyset$.
	Let $\xi_p\coloneq\max\nabla+1<\gamma$ and $\xi_q\coloneq\max\dom(q)+1\geq\xi_p$. By extending $q$ and increasing $m_0$, we may assume that:
	\begin{equation}
		\label{eq_I}
		\text{for }n\in n^\prime\setminus I, \sup_{m<\omega}\a_{n,m}<\xi_q \text{ and for }n\in I, \a_{n, m}>\xi_q \text{ for }m>m_0.
	\end{equation} 
	Since $\bar{p}_{\xi_p}\coloneq\langle p_m\on \xi_p:m<\omega\rangle$ also forms a uniform $\Delta$-system (in $\P_{\xi_p}$ and hence in $\P_{\xi_q}$) and $\lim\bar{p}_{\xi_p}=\lim\bar{p}=p^\infty$ by definition,
	there are $q^\prime\leq q$ in $\P_{\xi_q}$ and $m>m_0$ such that $q^\prime\Vdash_{\xi_q} p_m\on\xi_p\in\dot{G}_{\xi_q}$ by Induction Hypothesis. We may assume $q^\prime\leq p_m\on\xi_p$ and by \eqref{eq_I}, $r\coloneq q^\prime\cup\{\langle\a_{n,m},p_m(\a_{n, m})\rangle:n\in I\}$ is a valid condition in $\P_\gamma$ extending $q^\prime$.
	Again by \eqref{eq_I}, we have $\dom(p_m)=\dom(p_m\on\xi_q)\cup\{\a_{n,m}:n\in I\}$, so $r$ also extends $p_m$.%, which was what we wanted.
\end{proof}

Let us introduce the following forcing notions $\pst$ and $\pstt$, which generically add winning strategies and therefore increase $\sIstar$ and $\sIwstar$, respectively:
\begin{defi}
	\begin{itemize}
		\item $\fstr\coloneq \bigcup_{n<\omega}2^{(2^{<n})}$, the set of all finite partial strategies.
		\item $\pstt\coloneq\{(\sigma,F):\sigma\in\fstr, F\in[2^\omega]^{<\omega}\}$. $(\sigma^\prime, F^\prime)\leq(\sigma,F):\Leftrightarrow \sigma^\prime\supseteq\sigma, F^\prime\supseteq F$ and for all $n\in\left[|\sigma|,|\sigma^\prime|\right)$ and $y\in F$, $\sigma^\prime(y\on n)\leq y(n)$.
		\item $\pst\coloneq\{(\sigma,F)\in\pstt: F\subseteq 2^\omega\setminus\zero \}$ and the order is defined by restriction.
		\item For both $\pst$ and $\pstt$, $\sigma_G$ denotes the generic strategy $\sigma_G\coloneq\bigcup_{(\sigma,F)\in G}\sigma$ for a generic filter $G$. 
		
	\end{itemize}
	
\end{defi}

\begin{lem}
			\begin{enumerate}
				\item $\pst$ and $\pstt$ are $\sigma$-centered.
				\item For $y\in 2^\omega$, $\Vdash_{\pstt} y\rp_0\sigma_G$, where $\rp_0\coloneq\bigcup_{n<\omega}\rp_{0,n}$.
				\item For $y\in 2^\omega\setminus\zero$, $\Vdash_{\pst} y\rp_0\sigma_G$.

				\item For $y\in 2^\omega\setminus\zero$, $\Vdash_{\pstt} \sigma_G*y\in2^\omega\setminus\zero$.
				\item For $y\in 2^\omega$, $\Vdash_{\pst} \sigma_G*y\in2^\omega\setminus\zero$.
			\end{enumerate}
		\end{lem}
		\begin{proof}
		    The first three items are easy to prove, so we show the remaining two. In both cases, for $y\in2^\omega\setminus\zero$ and $m<\omega$, there are densely many $(\sigma,F)$ satisfying that there is $n>m$ such that $y(n)=1$, $y\on n\neq z\on n$ for all different $z\in F$, $|\sigma|>n$ and $\sigma(y\on n)=1$.
		In the case of $\pst$, 
		for $y\in\zero$ and $m<\omega$
		there are densely many $(\sigma,F)$ satisfying that there is $n> m$ such that $y\on n\neq z\on n$ for all $z\in F$, $n<|\sigma|$ and $\sigma(y\on n)=1$, since $y\notin F$. Now, we are done.
		\end{proof}

		\begin{cor}
			For $y\in2^\omega$, $\Vdash_{\pstt} y\trii \sigma_G$ and $\Vdash_{\pst} y\tri \sigma_G$.
			Hence, by iteration, $\pst$ and $\pstt$ increase $\sIstar$ and $\sIwstar$, respectively.
		\end{cor}

  \begin{rmk}
      $\pst$ and $\pstt$ are not the exact posets trying to add a generic strategy $\sigma_G$ such that $y\rp\sigma_G$ for all $y\in2^\omega$, but they are for $y\rp_0\sigma_G$ and in this sense they are defined naturally to some extent.
      We are not sure if there are exact natural ones for $y\rp\sigma_G$. %as there is not a ``natural'' forcing known that adds a reaping real and hence increases $\fraks$. 
  \end{rmk}

\begin{lem}
\label{lem:pstt_has_Frlimits}
	$\pstt$ is $\sigma$-$\mathrm{Fr}$-linked.
	%Moreover, $\sigma$-$\Lambda^\mathrm{lim}_\mathrm{uf})$-linked.
\end{lem}

To prove this lemma, we show a stronger property using ultrafilters introduced in \cite{goldstern2016left}, since using ultrafilters will make it easier to understand what is going on for $\pst$ and $\pstt$ (see also Remark \ref{rmk:compact}).
\begin{defi}
    Let $D$ be a non-principal ultrafilter on $\omega$ and $\P$ be a poset. $Q\subseteq \P$ is $D$-lim-linked if there exist a function $\lim^D\colon Q^\omega\to\P$ and a $\P$-name $\dot{D}^\prime$ of an ultrafilter extending $D$ such that for any countable sequence $\bar{q}=\langle q_m:m<\omega\rangle\in Q^\omega$, 
	\begin{equation}
		\textstyle{\lim^D\bar{q}} \Vdash \dot{W}(\bar{q})\in\dot{D}^\prime.
	\end{equation}
 Note that if $Q\subseteq \P$ is $D$-lim-linked, then $Q$ is Fr-linked.
\end{defi}

The fourth author gave a forcing-free characterization of $D$-lim-linkedness in \cite{Yam24}:
\begin{fact}[{\cite[Lemma 3.28]{Yam24}}]
    \label{fact:chara_UF_linked}
			Let $D$ be a non-principal ultrafilter on $\omega$ and $\P$ be a poset, $Q\subseteq\P$, $\lim^D\colon Q^\omega\to\P$.
			Then, the following are equivalent:
			\begin{enumerate}
				\item $\lim^D$ witnesses $Q$ is $D$-lim-linked.\label{item_lim_1}
				\item $\lim^D$ satisfies $(\star)_n$ below for all $n<\omega$:\label{item_lim_2}
				\begin{align*}
					(\star)_n:& \text{``Given }\bar{q}^j=\langle q_m^j:m<\omega\rangle\in Q^\omega\text{ for }j<n\text{ and }r\leq\textstyle{\lim^D}\bar{q}^j\text{ for all }j<n,\\
					&\text{then }\{m<\omega:r \text{ and all }q_m^j \text{ for } j<n \text{ have a common extension}\}\in D\text{''}.
				\end{align*}
			\end{enumerate}
\end{fact}

\begin{proof}[Proof of Lemma \ref{lem:pstt_has_Frlimits}]
	Fix $\sigma\in\fstr$ and $k<\omega$.
	Let $Q=Q_{\sigma,k}\coloneq\{(\sigma^\prime, F^\prime)\in\pstt:\sigma^\prime=\sigma, |F|\leq k\}$ and $D$ a non-principal ultrafilter on $\omega$.
	We show $Q$ is $D$-lim-linked.
	Let $\bar{q}=\langle q_m=(\sigma,\{y^m_i:i<k\})\rangle_{m<\omega}\in Q^\omega$.
	For $i<k$, define $y^\infty_i\in2^\omega$ by:
	\begin{equation}
		y^\infty_i(n)=j:\Leftrightarrow\{m<\omega:y^m_i(n)=j\}\in D,
	\end{equation}
    for $n<\omega$ and $j<2$.
	Let $q^\infty=\lim^D\bar{q}\coloneq(\sigma,\{y^\infty_i:i<k\})$.
	We first show $(\star)_1$ and then $(\star)_n$ for $n\geq2$ in Fact \ref{fact:chara_UF_linked}.
 
	\textit{Case $(\star)_1$}: Let $r=(\sigma^\prime,F^\prime)\leq q^\infty=(\sigma,\{y^\infty_i:i<k\})$.
		Fix $n\in\left[|\sigma|,|\sigma^\prime|\right)$ and $i<k$.
		Since $r\leq q^\infty$, $\sigma^\prime(y^\infty_i\on n)(n)\leq y^\infty_i(n)$ holds.
		Let $X_{n,i}\coloneq\{m<\omega:y^\infty_i\on(n+1)=y^m_i\on(n+1)\}\in D$.
		Unfix $n,i$ and let $X\coloneq\bigcap_{n,i}X_{n,i}\in D$.
		For any $m\in X$, $n\in\left[|\sigma|,|\sigma^\prime|\right)$ and $i<k$, $\sigma^\prime(y^m_i\on n)(n)=\sigma^\prime(y^\infty_i\on n)(n)\leq y^\infty_i(n)=y^m_i(n)$, so $r_m\coloneq(\sigma^\prime,F_m\cup F^\prime)$ is a common extension of $r$ and $q_m$ for $m\in X$.
		
	\textit{Case $(\star)_n$}: Let $n\geq2$ and assume $(\star)_1$ holds.
		Let $r=(\sigma^\prime,F^\prime)$ and $q_m^j=(\sigma,F^j_m)$ for $m<\omega$ and $j<n$.
		By $(\star)_1$, $Y_j\coloneq\{m<\omega:r\text{ and }q^j_m \text{ have a common extension }\}\in D$ for $j<n$.
		For $m\in\bigcap_{j<n}Y_j(\in D)$, $r_m\coloneq(\sigma^\prime,\bigcup_{j<n}F^j_m\cup F^\prime)$ is a common extension of $r$ and all $q^j_m$ for $j<n$. 
\end{proof}

\begin{rmk}\label{rmk:compact}
	In the proof above, $y^\infty_i$ is the $D$-limit point of the sequence $\langle y^m_i:m<\omega\rangle$ in Cantor space $2^\omega$ in the topological sense: for any neighborhood $U$ of the limit point $y^\infty_i$, $\{m<\omega:y^m_i\in U\}\in D$ holds.
	Note that every compact (Hausdorff) space has a (unique) $D$-limit (see e.g. \cite[Lemma 3.6]{brendle2021filter}).
	In the case of $\pst$, $2^\omega\setminus\zero$ is not compact and hence the same proof does not work (Actually, Main Lemma \ref{mainlem} implies $\pst$ is not $\sigma$-$\mathrm{Fr}$-linked).
\end{rmk}

%\begin{dfn}
%	For $j\in 2$ and $i\leq\omega$, $j\star i$ denotes the sequence of length $i$ whose values are all $j$. Namely, $j\star i\coloneq i\times\{j\}$.
%\end{dfn}

Fr-limits keep $\sIstar$ small:
\begin{mainlem}
	\label{mainlem}
	Let $\kappa$ be uncountable regular, $\gamma>\kappa$ a limit ordinal and $\P_\gamma$ a $<\kappa$-Fr-iteration whose first $\kappa$-many iterands are Cohen forcings $\mathbb{C}$.
	Then, $\Vdash_{\P_\gamma}\sIstar\leq\kappa$.
\end{mainlem}

\begin{proof}
	We show that the first $\kappa$-many Cohen reals $\langle\dot{c}_\a:\a<\kappa\rangle$ (as members of  $2^\omega$)  witness $\Vdash_{\P_\gamma} \sIstar\leq\kappa$, by focusing on the characterization $\sIstar=\frakb(\rs_\infty)$.  Assume towards contradiction that there are $p\in\P_\gamma$ and a $\P_\gamma$-name $\dot{\sigma}$ of an element of $\str_\infty$ such that for all $\a<\kappa$, $p\Vdash \dot{c}_\a \tri\dot{\sigma} $, particularly $p\Vdash \dot{c}_\a \rp\dot{\sigma}$. 
Since $\rp=\bigcup_{j\in2}\bigcup_{n<\omega}\rp_{j,n}$,  for $\a<\kappa$ we obtain  $p_\a\leq p$, %$\b_\a\in\left[i,\kappa\right)$, 
$j_\a\in 2$ and $n_\a<\omega$ such that $p_\a\Vdash \dot{c}_\a\rp_{j_\a,n_\a}\dot{\sigma}$. 
By extending and thinning, we may assume there is $I\in[\kappa]^{\kappa}$ such that: 

\begin{enumerate}
	\item $\a\in\dom(p_\a)$  for $\a\in I$. %(By extending $p_\a$.)
	%\item All $ p_\a$ follow a common guardrail $h\in H$. ($|H|<\kappa$.)
	\item $\{p_\a:\a\in I\}$ forms a uniform $\Delta$-system with root $\nabla $. %(By Lemma \ref{lem_uniform_Delta_System_Lemma}.)
	%\item $\b_\a\notin\nabla$, hence all $\b_\a$ are distinct. ($\b_\a$ are unbounded in $\kappa$, so $\b_\a$ are eventually out of the finite set $\nabla$.)
	\item All $j_\a$ are equal to $ j$ and all $n_\a$ are equal to $ n^*$.
	\item All $p_\a(\a)$ are the same Cohen condition $s\in\sqtwo$. 
	\item $|s|=n^*$. (By increasing $n^*$ or extending $s$.)
\end{enumerate}

In particular, we have that:
\begin{equation}
	\label{eq_property_of_refined_pi}
	\text{For each }\a\in I, ~p_\a\text{ forces }\dot{c}_{\a}\on n^*=s\text{ and }\dot{c}_{\a}\rp_{j,n^*}\dot{\sigma}.
\end{equation}

Pick some countable $\{\a_0<\a_1<\cdots\}\in[I\setminus\nabla]^\omega$. 
For $m<\omega$,
define $q_m\leq p_{\a_m}$ by extending the $\a_m$-th position to $q_m(\a_m):=s^\frown\langle j\rangle^m\in2^{n^*+m}$.
By \eqref{eq_property_of_refined_pi}, for all $m<\omega$ we have:
\begin{equation}
	\label{eq_qi}
	q_m\Vdash \text{ for all }n<m,~\dot{\sigma}(s^\frown\langle j\rangle^n)=0.
\end{equation}

Since $\a_m\notin\nabla$, $\bar{q}\coloneq\langle q_m:m<\omega\rangle$ also forms a uniform $\Delta$-system (with root $\nabla$) and hence we can take the limit $q^\infty\coloneq\lim\bar{q}$.

Let $y\coloneq s^\frown \langle j\rangle^\omega\in\zero\cup\one$. 
Since $q^\infty\Vdash \exists^\infty m<\omega~q_m\in\dot{G}$ and by \eqref{eq_qi}, we have:
\begin{equation}
	\label{eq_contra}
	q^\infty\Vdash\dot{\sigma}*y\in\zero,
\end{equation}

which contradicts $\dot{\sigma}\in\str_\infty$.
\end{proof}

%\begin{rmk}
%	In splitting** games, \eqref{eq_contra} does not matter for Player I. 
%\end{rmk}

\begin{thm}
\label{thm:Con_sIstar<sIwstar}
    Let $\kappa\leq\lambda$ be uncountable regular cardinals. Then, it is consistent that $\kappa=\sIstar\leq\sIwstar=\lambda$ holds.
\end{thm}
\begin{proof}
    Let $\mu>\lambda$ be such that $\mu^{<\kappa}=\mu$ and $\cf(\mu)=\lambda$ (e.g., construct the continuous sequence $\langle \mu_\a:\a\leq\lambda\rangle$ by $\mu_0\coloneq\lambda$ and $\mu_{\a+1}\coloneq(\mu_\a^{<\kappa})^+$ and put $\mu\coloneq\mu_\lambda$).
    Let $\P=\P_\gamma$ be a $<\kappa$-Fr-iteration of length $\gamma=\mu$ whose first $\kappa$-many iterands are Cohen forcings and each of the rest is either $\pstt$ or a subforcing of $\pst$ of size $<\kappa$ (by bookkeeping), which is possible by Lemma \ref{lem:pstt_has_Frlimits} (and Lemma \ref{lem:poset_is_its_size_Frlinked}). It is easy to see that $\P$ forces $\lambda\leq\sIwstar\leq\nonm\leq\lambda$ by $\cf(\mu)=\lambda$ and $\sIstar\geq\kappa$ by bookkeeping (and $\mu^{<\kappa}=\mu$). Main Lemma \ref{mainlem} implies $\Vdash_\P \sIstar\leq\kappa$.
\end{proof}

\begin{thm}\label{thm:separation_of_three}
    Let $\theta\leq\kappa\leq\lambda$ be uncountable regular cardinals such that $\kappa$ is $\aleph_1$-inaccessible, i.e., $\mu<\kappa$ implies $\mu^{\aleph_0}<\kappa$.
    Then, it is consistent that $\mathfrak{s}=\theta\leq\sIstar=\kappa\leq\sIwstar=\lambda$ holds.
\end{thm}

\begin{proof}
    Since $\pst$ and $\pstt$ are Suslin ccc posets, we can use the method in \cite{Goldstern2021}: We first add a splitting family $\mathcal{S}$
    of size $\theta$ by using the forcing notion defined in \cite[Section 4]{Goldstern2021}, and then force with (a forcing-equivalent poset of) $\P_\gamma$ in the previous proof (subforcings of $\pst$ are the form of \cite[Definition 5.1(S2)(iii)]{Goldstern2021} and here we use the $\aleph_1$-inaccessibility of $\kappa$), interweaving cofinally many steps where each iterand is the finite support product of all $\sigma$-centered posets of size $<\theta$ consisting of reals (as seen in \cite[Subsection 6.B(I2)(ii)]{Goldstern2021}).
    Then, since $\pst$ and $\pstt$ are Suslin ccc and by \cite[Theorem 4.19, Lemma 5.5, 5.6, see also Lemma 6.8]{Goldstern2021}, $\mathcal{S}$ stays a splitting family and we have $(\mathfrak{p}=)\mathfrak{s}=\theta\leq\sIstar=\kappa\leq\sIwstar=\lambda$ in the final extension (see \cite{Goldstern2021} for details).%Then, the splitting family $\mathcal{S}$ is preserved through a finite support iteration of ``restricted'' (see \cite[Definition 5.1]{Goldstern2021}, and here we use the $\aleph_1$-inaccessibility of $\kappa$) Suslin ccc posets and ccc product forcings with finite supports of posets of size $<\theta$ (\cite[Theorem 4.19, Lemma 5.5, 5.6]{Goldstern2021}), so we have $(\mathfrak{p}=)\mathfrak{s}=\theta\leq\sIstar=\kappa\leq\sIwstar=\lambda$ in the forcing extension 
\end{proof}

\begin{rmk}\label{rmk:preservingssigma}
    In the previous extension model, %\ref{thm:separation_of_three},
    $\mathfrak{s}_\sigma=\mathfrak{s}=\theta$ holds, since the splitting family $\mathcal{S}$ above satisfies that for any $a\in\ooo$, $\{b\in\mathcal{S}: b \text{ does not split }a\}$ has size $<\theta$ (\cite[Theorem 4.19]{Goldstern2021}) and hence $\mathcal{S}$ is also a $\sigma$-splitting family.
    %added by the first forcing also forms a $\sigma$-splitting family preserved thorugh the remaining forcing as well.
\end{rmk}

%\begin{rmk}\label{rmk:preservingssigma}
 %       Using a forcing notion defined Section 4 of \cite{Goldstern2021} before we force by the finite support iteration of $P$ defined above theorem, we can also force that $\fraks_\sigma < \sIstar$.
 %       This is because the forcing in \cite{Goldstern2021} adds $\sigma$-splitting family of size $\aleph_1$ that is not destroyed by finite support iteration of Suslin ccc forcing notions.
 %   \end{rmk}

 Thus, we can say that the hidden structure behind the consistency of $\sIstar<\sIwstar$ is the \textit{compactness} of $\pstt$ as opposed to the \textit{incompactness} of $\pst$ mentioned in Remark \ref{rmk:compact}. As we have seen in Section \ref{sec:intro}, the statement of $\operatorname{Con}(\sIstar<\sIwstar)$ itself sounds slightly strange, but it turns out that the reason why this consistency holds is a consequence of the following obvious fact: \textit{$2^\omega$ is compact but $2^\omega\setminus\zero$ is not.} 
 We will consider this phenomenon again in Question \ref{ques:pair} in the following section.

\section{Questions and discussions}\label{sec:que}
\begin{question}
        Does $\ZFC$ prove $\sIstar \le \non(\mathcal{E})$ (or $\sIwstar \le \non(\mathcal{E})$), where $\mathcal{E}$ denotes the $\sigma$-ideal generated by closed null sets?
    \end{question}

    %begin{question}
     %   What is the value of $\sIstar$ in the model obtained by finite support iteration of the random forcing over a model of $\CH$? (Note that in this model $\non(\mathcal{E})$ is small and $\non(\meager)$, $\frakd$ and $\non(\nul)$ are large.)
   % \end{question}
    %(essentially solved by Main Lemma \ref{mainlem})

    \begin{question}
        Is there a lower bound of $\sIstar$ other than $\fraks_\sigma$ and $\min\{\sIwstar,\frakb\}$? In particular, is $\add(\nul)$ a lower bound of $\sIstar$?
    \end{question}

    \begin{question}
        Does $\ZFC$ prove $\sIwstar\leq\frakd$?
    \end{question}
    Even though Theorem \ref{thm:sIstar_le_frakd} shows $\sIstar\leq\frakd$, it seems that its proof cannot be easily modified to that of $\sIwstar\leq\frakd$, since the proof relies on the characterization using $\str_\infty$, but not $\str_\infty^\one$.
    %One might be able to see this result to reason that $\sIstar$ is more ``natural'' combination of the notions of \textit{splitting} and \textit{game}. 

    \begin{question}
        What is the relationship among the dominating numbers of the four relational systems $\rs, \rss, \rs_\infty, \rss_\infty$?
    \end{question}
    We have the following observations:

    \begin{lem}
        \begin{enumerate}
            \item $\frakd(\rs)\leq\frakd(\rs_\infty)$.
            \item $\frakd(\rss)\leq\frakd(\rss_\infty)$.
            \item $\frakd(\rss_\infty)\leq\frakd(\rs_\infty)$.
            \item $\frakd(\rs)=\frakd(\rss)$.
        \end{enumerate}
    \end{lem}
    \begin{proof}
        The first three items and $\frakd(\rs)\geq\frakd(\rss)$ are easy to prove, so we show $\frakd(\rs)\leq\frakd(\rss)$. Given an $\rss$-dominating family $F\subseteq\str$, $F^\prime\coloneq F\cup\{\sigma_\one\}$ is $\rs$-dominating where $\sigma_\one$ is the strategy such that $\sigma_\one(y)=1$ for all $y\in\sqtwo$, and $F^\prime$ has size $|F|$ since $\frakd(\rss)$ is easily proved to be infinite.
    \end{proof}
    Hence, the inequalities are summarized as: $\frakd(\rs)=\frakd(\rss)\leq\frakd(\rss_\infty)\leq\frakd(\rs_\infty)$, while in the case of the bounding numbers we have $\sIstar=\frakb(\rs_\infty)=\frakb(\rs)\leq\frakb(\rss)=\frakb(\rss_\infty)=\sIwstar$.
    Moreover, by considering the ``dual'' case of Theorem \ref{thm:Con_sIstar<sIwstar}, we have:
    \begin{lem}
        Consistently, $\frakd(\rs)=\frakd(\rss)=\frakd(\rss_\infty)<\frakd(\rs_\infty)$ holds.
    \end{lem}
    Thus, the remaining open relationship is:
    \begin{subque}
        Does $\ZFC$ prove $\frakd(\rss)\geq\frakd(\rss_\infty)$?
    \end{subque}

    \begin{question}\label{ques:pair}
        Is there another example of a pair of cardinal invariants that have the same relationship as $\sIstar$ and $\sIwstar$?
    \end{question}

    $\sIstar$ and $\sIwstar$ have the following features:
    \begin{enumerate}
        \item The corresponding poset $\pstt$ of $\sIwstar$ has compactness (and hence UF-limits).
        \item The corresponding poset $\pst$ of $\sIstar$ does not have compactness (or UF-limits), and Fr-limits (and UF-limits) keep $\sIstar$ small.
        \item Thus, $\sIstar<\sIwstar$ consistently holds. 
    \end{enumerate}
    We are not sure if there is another (non-artificial) example of a pair of two numbers such that their difference lies in whether their corresponding forcing notions have compactness or not, and consequently they are consistently different.

%\section{Acknowledgements}

%    The authors are grateful to J\"{o}rg Brendle, Osvaldo G\'uzman and Michael Hru\v{s}\'ak for their helpful comments.
%    Remark \ref{rmk:preservingssigma} is based on a suggestion from Diego A. Mejía.
%    This work was supported by JSPS KAKENHI Grant Number JP22J20021 and JST SPRING, Japan Grant Number JPMJSP2148.

	\printbibliography

\end{document}